\documentclass[a4paper,10pt]{article}

\usepackage[english]{babel}
\usepackage{amsmath,amsfonts,amssymb,amsthm}
\usepackage{mathrsfs}
\usepackage{graphicx}
\usepackage{bbm}
\usepackage{indentfirst}

\usepackage{relsize}
\usepackage{pgf,tikz}
\usetikzlibrary{arrows}
\usetikzlibrary{patterns}
\usepackage[all]{xy}

\newtheorem{theorem}{Theorem}[section]

\newtheorem{corollary}{Corollary}[section]
\theoremstyle{definition}
\newtheorem{example}{Example}[section]
\theoremstyle{definition}
\newtheorem{remark}{Remark}[section]

\numberwithin{equation}{section}

\newcommand\blfootnote[1]{\begingroup\renewcommand\thefootnote{}\footnote{#1}\addtocounter{footnote}{-1}\endgroup}

\DeclareTextCompositeCommand{\u}{PD1}{\i}{i}

\begin{document}

\title{
{\bf\Large
A note on a fixed point theorem \\ on topological cylinders }\footnote{Work performed under the auspices of the
Grup\-po Na\-zio\-na\-le per l'Anali\-si Ma\-te\-ma\-ti\-ca, la Pro\-ba\-bi\-li\-t\`{a} e le lo\-ro
Appli\-ca\-zio\-ni (GNAMPA) of the Isti\-tu\-to Na\-zio\-na\-le di Al\-ta Ma\-te\-ma\-ti\-ca (INdAM).}}

\author{
{\bf\large Guglielmo Feltrin}
\vspace{1mm}\\
{\it\small SISSA - International School for Advanced Studies}\\
{\it\small via Bonomea 265}, {\it\small 34136 Trieste, Italy}\\
{\it\small e-mail: guglielmo.feltrin@sissa.it}\vspace{1mm}}

\date{}

\maketitle

\vspace{-2mm}

\begin{abstract}
\noindent
We present a fixed point theorem on topological cylinders in normed linear spaces 
for maps satisfying a property of stretching a space along paths. 
This result is a generalization of a similar theorem obtained by D.~Papini and F.~Zanolin. 
In view of the main result we discuss the existence of fixed points for maps defined on different types of domains
and we propose alternative proofs for classical fixed point theorems, as Brouwer, Schauder and Krasnosel'ski\u{\i} ones.
\end{abstract}

\section{Introduction}\label{section1}

Fixed point theorems in finite and infinite dimensional spaces have many applications, for instance, 
to the search of positive solutions of integral equations, to the study of periodic ordinary differential equations and in chaos theory.
We mention \cite{AvAnHe-2011,ChTo-2007,ErHuWa-1994,ErWa-1994,FrTo-2008,GrKoWa-2008,LeWi-1980,PiZa-2005,PiZa-2007,Pr-2007,To-2003}, as examples.

Recently, with respect to this issue, fixed point theorems has received lot of attention.
Accordingly, in several works classical results has been reinterpreted.
For example, in \cite{Kw-2008fpta,Kw-2008} the author started from Krasnosel'ski\u{\i} cone fixed point theorem to present more general results in cones and cylinders. 
Or still, in \cite{Mar-2006} there is an application of a general version of the classical Poin\-ca\-r\'e-Miranda zeros theorem, 
which is equivalent to Brouwer fixed point theorem, as well known.
As a last example, we mention \cite{Ma-2013}, where the author has proposed elementary proofs of generalization of fixed point theorems in $\mathbb{R}^{n}$, which are
useful tools for the study of ordinary differential equations.

\medskip

The present paper deals with the search and the localization of fixed points for compact maps defined on domains of a normed linear space.
Even if there are more applications in $\mathbb{R}^{n}$ than in infinite dimensional spaces, we consider maps defined on very general topological spaces.
Our aim is to obtain new fixed point theorems that would allow to achieve better criteria for the existence and multiplicity of solutions 
to nonlinear differential and integral equations.

In more detail, our starting point for the present paper is the work of D.~Papini and F.~Zanolin, \cite{PaZa-2007}.
The two authors have studied fixed point theorems for maps satisfying a property of stretching a space along paths connecting two of its own disjoint subsets.
In \cite{PaZa-2007} they have considered cylinders of the form 
\begin{equation*}
\mathopen{[}a,b\mathclose{]}\times B[0,R],
\end{equation*}
so with ``base'' a ball in $\mathbb{R}^{n}$, while our goal is to work with domains of the form
\begin{equation*}
\mathopen{[}a,b\mathclose{]}\times A,
\end{equation*}
where $A$ is an \textit{absolute retract}. 
Afterwards we prove that the existence of fixed points remains valid under homeomorphic transformation of the cylinder.

\medskip

The plan of the paper is the following. In Section~\ref{section2} we present our main results. 
First of all we generalize \cite[Theorem~6]{PaZa-2007} by considering a cylinder with
an absolute retract as ``base'' and a compact map defined over it. Under the assumption of ``stretching along the paths'', 
we obtain an existence result for fixed points in a particular region of the cylinder, using the fixed point index and in the same spirit of \cite{PaZa-2007}.
Subsequently, we discuss further possible generalization of the theorem and
we verify that the property of ``stretching along the paths'' is preserved under homeomorphisms.

In Section~\ref{section3} we review some classical fixed point theorems, as Brouwer and Schauder theorems, from the perspective of the main result. 
Using simple tricks, we show that this results are direct applications of our main theorem.

Section~\ref{section4} is devoted to the study of maps defined on domains of different type. In more detail, we consider
particular region of cones in normed linear spaces, as in Krasnosel'ski\u{\i}'s cone fixed point theorem.
As we shall see, the hypotheses of our main result are clearly of expansive type, hence it allows us to analyze only the expansive form of Krasnosel'ski\u{\i} theorem. 
For this reason, in this section we introduce a corollary useful to demonstrate Krasnosel'ski\u{\i}'s compressive form and other compressive-type fixed point results.
In addition, we explore other fixed point theorems related to the main result.

We conclude our work with an appendix dedicated to the fixed point index. We recall the main properties and we enunciate 
the Leray-Schauder continuation principle. This tools are used during the proof of the main theorem in Section~\ref{section2}.

\medskip

Now we introduce some symbols and some notations. We denote by $\mathbb{Q}$ and $\mathbb{R}$ the usual numerical sets
and let $\mathbb{R}^{+}:=\mathopen{[}0,+\infty\mathclose{[}$.

Given a topological space $(X,\tau)$ and a subset $A\subseteq X$, we denote by $\overline{A}$ its closure, 
with $\text{\rm int}(A)$ its inner part and with $\partial A$ its boundary. 
If $A,B\subseteq X$, with $A\setminus B$ we mean the relative complement of $B$ in $A$.

What is dealt in the article takes place in the class of metrizable spaces. 
Usually we denote by  $(X,\|\cdot\|)$ or simply by $X$ a normed linear space, 
where $\|\cdot\|$ is its norm. 
The symbols $B(x_{0},r)$ and $B[x_{0},r]$, where $x_{0}\in X$ and $r>0$, represent the open and closed ball of centre $x_{0}$ and radius $r$ respectively, i.e.~
\begin{equation*}
B(x_{0},r):=\bigl{\{}x\in X\colon \|x-x_{0}\|<r\bigr{\}}, \quad B[x_{0},r]:=\bigl{\{}x\in X\colon \|x-x_{0}\|\leq r\bigr{\}}.
\end{equation*}

We indicate with
\begin{equation*}
\text{dist}(A,B):=\inf_{x\in A, \, y\in B} \|x-y\|
\end{equation*}
the usual distance between two sets $A,B\subseteq X$.

We denote by $Id$ or $Id_{X}$ the identity on the space $X$. Given a function $f$, 
$f|_{D}$ represents the restriction of the function $f$ in $D$, where $D$ is a subset of the domain of $f$.

\section{Main results}\label{section2}

In this section we present our main result (Theorem~\ref{ThPapiniZanolinretract}) which deals with cylinders having a topological retract as ``base''.  
Then we prove the validity of the same theorem for other type of domains.

Let $(X,\|\cdot\|)$ be a normed linear space and let $D$ be a subset of $\mathbb{R}\times X$. Consider a map
\begin{equation*}
\phi=(\phi_{1},\phi_{2})\colon D\to\mathbb{R}\times X.
\end{equation*}
Our main goal is to provide an existence result for fixed points of $\phi$. 
In other words, we want to prove the existence of elements
$\tilde{z}=(\tilde{t},\tilde{x})\in D$ such that
\begin{equation*}
\begin{cases}
\, \tilde{t}=\phi_{1}(\tilde{t},\tilde{x})\\
\, \tilde{x}=\phi_{2}(\tilde{t},\tilde{x}). 
\end{cases}
\end{equation*}

Now we introduce some preliminary definitions and notations.
First of all we recall that a nonempty subset $A$ of a topological space $Y$ is a \textit{retract} of $Y$ if there exists a
continuous map $r\colon Y\to A$ (called \textit{retraction}), such that $r|_{A}=Id_{A}$.

Let $A$ be a retract of $X$ and fix $a,b\in\mathbb{R}$ with $a<b$. We denote by
\begin{equation*}
\mathcal{C}:=\mathopen{[}a,b\mathclose{]}\times A
\end{equation*}
the \textit{cylinder} with ``base'' the retract $A$ and ``height'' the interval $\mathopen{[}a,b\mathclose{]}$.
We also set
\begin{equation*}
\mathcal{C}_{l}:=\{a\}\times A \quad \text{ and } \quad \mathcal{C}_{r}:=\{b\}\times A
\end{equation*}
the cylinder's left base and the cylinder's right base respectively.

In the following we will put our attention on \textit{paths} contained in $\mathcal{C}$ or in other subsets of $\mathbb{R}\times X$. 
With the term \textit{path} we mean an element $(\sigma,\overline{\sigma})$, where $\sigma$ is a continuous map defined on a compact interval 
$I\subseteq\mathbb{R}$ (usually we take $I=\mathopen{[}0,1\mathclose{]}$) and with values on the normed linear space $\mathbb{R}\times X$, 
and $\overline{\sigma}:=\sigma(I)\subseteq \mathbb{R}\times X$ is the \textit{support} (image) of $\sigma$. Furthermore, 
we say that $(\gamma,\overline{\gamma})$ is a \textit{sub-path} of the path $(\sigma,\overline{\sigma})$ (defined on $I$) 
if $\gamma\colon\mathopen{[}s_{0},s_{1}\mathclose{]}(\subseteq I)\to \mathbb{R}\times X$ and $\gamma=\sigma|_{\mathopen{[}s_{0},s_{1}\mathclose{]}}$.

In order to avoid too heavy notation we simply employ $\sigma$ to denote the path $(\sigma,\overline{\sigma})$. 
In addition, we often use the symbol $\sigma$ meaning its support; for example, 
if $Z\subseteq \mathbb{R}\times X$, we write $\sigma\subseteq Z$ instead of $\overline{\sigma}\subseteq Z$, $\sigma\cap Z\neq\emptyset$ 
in place of $\overline{\sigma}\cap Z\neq\emptyset$ and so on. Moreover, if $\sigma$ and $\gamma$ are two paths, the notation $\gamma\subseteq\sigma$ 
means that $\gamma$ is a sub-path of $\sigma$.

As usual, a map $f\colon X\to Y$ between two metric spaces is \textit{compact}
if it is continuous and the closure of $f(X)$ is a compact subset of $Y$.

We now state and prove the main result.

\begin{theorem}\label{ThPapiniZanolinretract}
Let $(X,\|\cdot\|)$, $D$, $A$, $\mathcal{C}$ and $\phi$ be as above.
Suppose that there exists a closed set $W\subseteq D\cap \mathcal{C}$ such that
\begin{enumerate}
\item [$(i)$] $\phi$ is compact on $W$;
\item [$(ii)$] for every path $\sigma\subseteq \mathcal{C}$ with $\sigma\cap \mathcal{C}_{l}\neq\emptyset$ 
and $\sigma\cap \mathcal{C}_{r}\neq\emptyset$ there exists a sub-path $\gamma\subseteq\sigma\cap W$
with $\phi(\gamma)\subseteq \mathcal{C}$, $\phi(\gamma)\cap \mathcal{C}_{l}\neq\emptyset$ and $\phi(\gamma)\cap \mathcal{C}_{r}\neq\emptyset$.
\end{enumerate}
Then there exists $\tilde{z}\in W$ such that $\phi(\tilde{z})=\tilde{z}$.
\end{theorem}

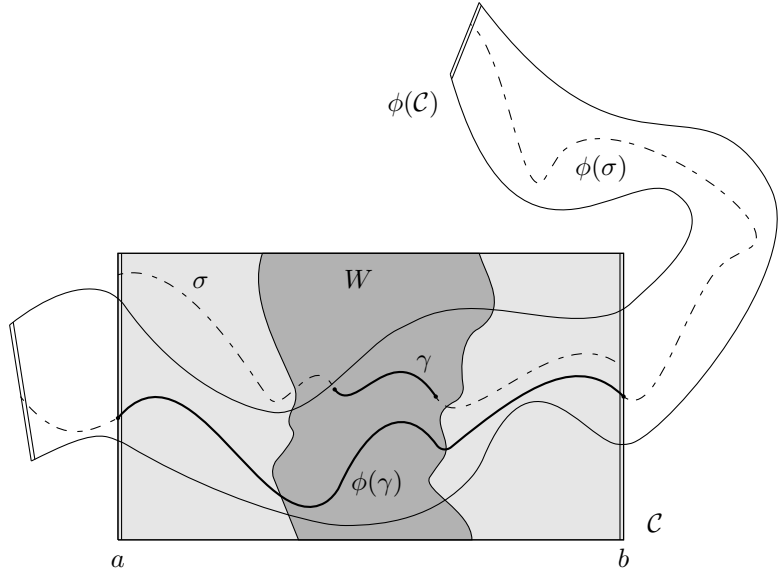
\begin{figure}[ht]
	\centering
\begin{tikzpicture}[x=0.95cm,y=0.95cm]
\filldraw[fill=gray!20] (0,0) -- (7,0) -- (7,4) -- (0,4) -- (0,0);
\draw (0.05,0) -- (0.05,4);
\draw (6.95,0) -- (6.95,4);
\filldraw[fill=gray!60]
											[rounded corners]		(2,4) .. controls (1.8,3.4) and (2.2, 2.7) .. 
															(2.5,2) .. controls (2.2,1.4) and (2.7,1.7) .. 
															(2.1,1) 
											[sharp corners] 		-- (2.5,0) -- (4.9,0)
											[rounded corners]		.. controls (4.9,0.4) and (4.2,0.7) ..
															(4.4,1)
															.. controls (4.7,1.4) and (4.2,1.7) .. 
															(4.8,2) 
															.. controls (4.7,2.7) and (4.9,2.9) ..
															(5.1,3) 
											[sharp corners]			 .. controls (5.3, 3.4) and (5.1,3.7)  .. 
															(5,4) -- (2,4);
\draw [rounded corners] (-1.5,3) .. controls (-1,3.4) and (-0.5, 3.6) .. 
												(0,3.4) .. controls (0.7,2.4) and (1.4, 2) .. 
												(2,1.8) .. controls (2.7,1.7) and (3.4,2.7) .. 
												(4,3) .. controls (5,3.5)  and (6,2.9) .. 
												(7,3.2) .. controls (7.3,3.5) and (8.2,4.2) .. 
												(7.8,4.8) .. controls (7,5.2) and (5.5,3.3) .. 
												(4.6,6.5);
\draw (-1.5,3) -- (-1.2,1.1);
\draw (-1.45,3.04) -- (-1.15,1.12);
\draw [rounded corners] (-1.2,1.1) .. controls (-0.8,1.3) and (-0.4, 1.5) .. 
												(0,1.4) .. controls (1,0.8) and (2,0.4) .. 
												(3,0.2) .. controls (3.7,0.2) and (4.4,0.3) .. 
												(5,1) .. controls (5.7,2.7)  and (6.4,1.5) .. 
												(7,1.3) .. controls (7.7,1.6) and (9.5,3.8) .. 
												(9,5) .. controls (8,6.5) and (7,4.8) .. 
												(5,7.5);
\draw (4.6,6.5) -- (5,7.5);
\draw (4.625,6.44) -- (5.03,7.46);
\draw (0,0) -- (7,0) -- (7,4) -- (0,4) -- (0,0);
\draw[dash pattern=on 2pt off 3pt on 4pt off 4pt] [rounded corners] 
												(0,3.7) .. controls (0.7,3.9) and (1.4,3) .. 
												(2,2.1) .. controls (2.3,1.7) and (2.7,2.7) ..
												(3,2.1);
\draw [line width=0.8pt] [rounded corners]
												(3,2.1) .. controls (3.3,1.7) and (3.8,2.9) ..
												(4.4,2); 
\draw[dash pattern=on 2pt off 3pt on 4pt off 4pt] [rounded corners] 
												(4.4,2) .. controls (4.8,1.3)  and (6.1,3.2) .. 
												(7,2.4);
\draw [dash pattern=on 2pt off 3pt on 4pt off 4pt] [rounded corners]
												(-1.35,2) .. controls (-0.8,1.3) and (-0.4, 1.5) ..
												(0,1.7); 
\draw [line width=0.8pt] [rounded corners](0,1.7) .. controls (1,2.9) and (2,0) .. 
												(3,0.6) .. controls (3.5,1.7) and (4,1.9) .. 
												(4.5,1.2) .. controls (5.7,2.2)  and (6.4,2.6) ..
												(7,2); 
\draw [dash pattern=on 2pt off 3pt on 4pt off 4pt] [rounded corners]
												(7,2) .. controls (7.5,1.6) and (7.9,3.6) .. 
												(8.6,4) .. controls (9.4,4.4) and (6.7,6.3) ..
												(6,5.2) .. controls (5.6,4.5) and (5.4,6.8) .. 
												(4.875,7.2);
\draw (7.2,0.5) node[anchor=north west] {$\mathcal{C}$};	
\draw (3.6,6.4) node[anchor=north west] {$\phi(\mathcal{C})$};	
\draw (3,3.9) node[anchor=north west] {$W$};
\draw (0.9,3.8) node[anchor=north west] {$\sigma$};
\draw (4,2.7) node[anchor=north west] {$\gamma$};
\fill (3,2.1) circle (0.8pt);
\fill (4.4,2) circle (0.8pt);	
\draw (6.2,5.5) node[anchor=north west] {$\phi(\sigma)$};											
\draw (3.1,1.1) node[anchor=north west] {$\phi(\gamma)$};
\fill (0,1.7) circle (0.8pt);
\fill (7,2) circle (0.8pt);	
\draw (0,-0.3) node {$a$};											
\draw (7,-0.25) node {$b$};	
\end{tikzpicture}
\caption{\small{The figure shows a possible situation satisfying conditions $(i)$ and $(ii)$ of Theorem~\ref{ThPapiniZanolinretract}. In the picture it is assumed that $D=\mathcal{C}$.}}
\label{PZ}
\end{figure}

\begin{proof}
Consider the compact operator
\begin{equation*}
\phi\colon W\to \mathbb{R}\times X.
\end{equation*}
As first step we extend $\phi$ to $\mathbb{R}\times X$ via a compact map
\begin{equation*}
\tilde{\phi}\colon\mathbb{R}\times X\to \mathbb{R}\times X,
\end{equation*}
i.e.~$\tilde{\phi}|_{W}=\phi|_{W}$. The existence of $\tilde{\phi}$ compact is guaranteed by an application of Urysohn embedding theorem 
(see, for instance, \cite[Theorem~B.11, p.~597]{GrDu-2003}) 
and Dugundji extension theorem (cf.~\cite{Du-1951} and \cite[Theorem~7.4, p.~163--164]{GrDu-2003}) together with the results in \cite[\textsection~6.2, p.~116--118]{GrDu-2003}.

Let $r\colon X \to A$ be a retraction. We introduce the operator
\begin{equation*}
\psi=(\psi_{1},\psi_{2})\colon \mathbb{R}\times X\to\mathbb{R}\times A, 
\end{equation*}
defined as follows
\begin{equation*}
\psi_{1}(t,x):=\tilde{\phi}_{1}(t,x), \qquad \psi_{2}(t,x):=r(\tilde{\phi}_{2}(t,x)).
\end{equation*}
We observe that $\psi$ is a compact map (by the compactness of $\tilde{\phi}$ and the continuity of $r$). We define
\begin{equation*}
\hat{A}:=\overline{\psi_{2}(\mathbb{R}\times X)}.
\end{equation*}
The set $\hat{A}$ is nonempty, compact and contained in $A$ (since $A$ is closed, by \cite[Theorem~5.1, p.~18]{Hu-1965}).

We also note that if $\tilde{z}=(\tilde{t},\tilde{x})$ is a fixed point of $\psi$ with $\tilde{z}\in W$ and $\phi_{2}(\tilde{z})\in A$, then
\begin{equation*}
\tilde{t}=\psi_{1}(\tilde{z})=\tilde{\phi}_{1}(\tilde{z})=\phi_{1}(\tilde{z})
\end{equation*}
and
\begin{equation*}
\tilde{x}=\psi_{2}(\tilde{z})=r(\tilde{\phi}_{2}(\tilde{z}))=r(\phi_{2}(\tilde{z}))=\phi_{2}(\tilde{z}),
\end{equation*}
hence $\tilde{z}\in W$ is a fixed point of $\phi$.

Now we study the fixed point problem
\begin{equation*}
x=\psi_{2}(t,x), \quad x\in X,
\end{equation*}
where $t$ is treated as a parameter in $\mathopen{[}a,b\mathclose{]}$.
From the definition of $\hat{A}$, it is obvious that $\psi_{2}(t,x)\in \hat{A}\subseteq A$ for all
$(t,x)\in\mathopen{[}a,b\mathclose{]}\times X$. We fix $r>0$ such that the open ball $B(0,r)$ contains the closed and bounded set $\hat{A}$. It
follows that, for all $t\in\mathopen{[}a,b\mathclose{]}$,
\begin{equation*}
x-\psi_{2}(t,x)\neq0, \quad  \forall \, x\in\partial B(0,r).
\end{equation*}
Consequently, the fixed point index $i_{X}(\psi_{2}(t,\cdot),B(0,r))$ is well defined and constant, for $t\in\mathopen{[}a,b\mathclose{]}$.
Consider now the compact homotopy $h\colon\mathopen{[}0,1\mathclose{]}\times B[0,r]\to X$, $h_{\lambda}(x):=\lambda\psi_{2}(t,x)$. 
For every $t\in\mathopen{[}a,b\mathclose{]}$ we have that
\begin{equation*}
x-\lambda\psi_{2}(t,x)\neq0, \quad \forall \, \lambda\in\mathopen{[}0,1\mathclose{]}, \; \forall \, x\in\partial B(0,r),
\end{equation*}
since $\lambda\psi_{2}(t,x)\in\lambda\hat{A}\subseteq B(0,r)$ for every $\lambda\in\mathopen{[}0,1\mathclose{]}$ and $x\in X$.
 
By the homotopy invariance property, the integer $i_{X}(h_{\lambda},B(0,r))$ is constant with respect to $\lambda\in\mathopen{[}0,1\mathclose{]}$.
In particular, for every $t\in\mathopen{[}a,b\mathclose{]}$,
\begin{equation*}
i_{X}(\psi_{2}(t,\cdot),B(0,r))=i_{X}(0,B(0,r))=1.
\end{equation*}

Therefore we can apply the \textit{Leray-Schauder continuation principle} (Theorem~\ref{LScont_prin}), which ensures that the nonempty set
\begin{equation*}
\Sigma:=\bigl{\{}(t,x)\in\mathopen{[}a,b\mathclose{]}\times B(0,r) \colon x=\psi_{2}(t,x)\bigr{\}}
\end{equation*}
contains a compact and connected (\textit{continuum}) set $S$ such that 
\begin{equation*}
S\cap \mathcal{C}_{l}\neq\emptyset \quad \text{ and } \quad S\cap \mathcal{C}_{r}\neq\emptyset.
\end{equation*}
As consequence, $p_{1}(S)=\mathopen{[}a,b\mathclose{]}$, where $p_{1}\colon\mathbb{R}\times X\to \mathbb{R}$, $p_{1}(t,x)=t$, 
and for every open $\tilde{A}$ containing $\hat{A}$
\begin{equation*}
p_{2}(S)\subseteq \hat{A}\subseteq \tilde{A},
\end{equation*}
where $p_{2}\colon\mathbb{R}\times X\to X$, $p_{2}(t,x)=x$.

Fix an open and bounded set $\tilde{A}$ such that $\hat{A}\subseteq \tilde{A}$ (for example $\tilde{A}=B(0,r)$, as before) 
and fix $\varepsilon\in\mathbb{R}$ such that $0<\varepsilon<\text{dist}(\hat{A}, X\setminus \tilde{A})<+\infty$. By the continuity of the retraction $r$,
for every $x\in X$ we can find a positive number $\delta_{x}=\delta_{x}(\varepsilon)$, with $0<\delta_{x}<\varepsilon$, such that for every $y\in X$ with
$\|y-x\|<\delta_{x}$ it holds that $\|r(y)-r(x)\|<\varepsilon$.

Consider the open covering of $S$ defined as
\begin{equation*}
\bigl{\{}\,\mathopen{]}t-\varepsilon,t+\varepsilon\mathclose{[}\times B(x,\delta_{x})\colon (t,x)\in S \bigr{\}}.
\end{equation*}
Since $S$ is compact, we can extract a finite sub-covering
\begin{equation*}
\bigl{\{}\,\mathopen{]}t_{i}-\varepsilon,t_{i}+\varepsilon\mathclose{[}\times B(x_{i},\delta_{x_{i}})\colon i=1,\ldots,m \bigr{\}},
\end{equation*}
where $(t_{i},x_{i})\in S$, so $t_{i}\in\mathopen{[}a,b\mathclose{]}$ and $x_{i}\in \hat{A} \subseteq A$, for all $i\in\{1,\ldots,m\}$.

Now we set
\begin{equation*}
\mathcal{U}_{\varepsilon}
:=\bigcup_{i=1}^{m} \, \mathopen{]}t_{i}-\varepsilon,t_{i}+\varepsilon\mathclose{[}\times B(x_{i},\delta_{x_{i}}),
\end{equation*}
which is contained in $\mathopen{]}a-\varepsilon,b+\varepsilon\mathclose{[}\times \tilde{A}$, because
$0<\delta_{x_{i}}<\varepsilon<\text{dist}(\hat{A}, X\setminus \tilde{A})$, for all $i=1,\ldots,m$.
The set $\mathcal{U}_{\varepsilon}$ is open and connected (because $S\subseteq\mathcal{U}_{\varepsilon}$ 
and $(t_{i},x_{i})\in S\cap (\mathopen{]}t_{i}-\varepsilon,t_{i}+\varepsilon\mathclose{[}\times B(x_{i},\delta_{x_{i}}))\neq\emptyset$ for all
$i\in\{1,\ldots,m\}$). 
Hence $\mathcal{U}_{\varepsilon}$ is also arcwise connected.
Consequently there exists a continuous map
$\vartheta\colon\mathopen{[}0,1\mathclose{]}\to\mathcal{U}_{\varepsilon}$ such that $\vartheta(0)\in \mathcal{C}_{l}$ and $\vartheta(1)\in \mathcal{C}_{r}$.
Passing, if necessary, to a sub-path inside the cylinder and reparameterizing the curve, we can also assume that,
for all $s\in\mathopen{[}0,1\mathclose{]}$,
\begin{equation*}
\vartheta_{1}(s):=p_{1}(\vartheta(s))\in\mathopen{[}a,b\mathclose{]}.
\end{equation*}
Then we define a second curve $\zeta(s):=(\zeta_{1}(s),\zeta_{2}(s))$, $s\in\mathopen{[}0,1\mathclose{]}$, as follows
\begin{equation*}
\zeta_{1}(s):=p_{1}(\vartheta(s))=\vartheta_{1}(s), \qquad \zeta_{2}(s):=r(p_{2}(\vartheta(s)))=r(\vartheta_{2}(s)).
\end{equation*}
We claim that the curve $\zeta$ satisfies
\begin{itemize}
	\item $\zeta(s)\in\mathcal{V}_{\varepsilon}\cap \mathcal{C}$, for all $s\in\mathopen{[}0,1\mathclose{]}$;
	\item $\zeta(0)\in \mathcal{C}_{l}$ and $\zeta(1)\in \mathcal{C}_{r}$;
\end{itemize}
where
\begin{equation*}
\mathcal{V}_{\varepsilon}
:=\bigcup_{i=1}^{m}\,\mathopen{]}t_{i}-\varepsilon,t_{i}+\varepsilon\mathclose{[}\times B(x_{i},\varepsilon).
\end{equation*}
From the definition, we derive immediately that $\zeta(s)\in \mathcal{C}$ 
and $\zeta_{1}(s)=\vartheta_{1}(s)\in\mathopen{[}a,b\mathclose{]}\subseteq p_{1}(\mathcal{V}_{\varepsilon})$, for every $s\in\mathopen{[}0,1\mathclose{]}$.
Consequently, to prove the first property it is sufficient to verify that there exists $j\in\{1,\ldots,m\}$ such that $\zeta_{2}(s)\in
B(x_j,\varepsilon)$. Fix $s\in\mathopen{[}0,1\mathclose{]}$.  Let $k\in\{1,\ldots,m\}$ be such that $\vartheta_{2}(s)\in B(x_k,\delta_{x_k})$.
We distinguish two cases. If $\vartheta_{2}(s)\in A$, then $\zeta_{2}(s)=\vartheta_{2}(s)\in B(x_k,\delta_{x_k})\subseteq B(x_k,\varepsilon)$.  If
$\vartheta_{2}(s)\notin A$, then
\begin{equation*}
\|\zeta_{2}(s)-x_k\| =\|r(\vartheta_{2}(s))-r(x_k)\|<\varepsilon,
\end{equation*}
by the choice of $\delta_{x}=\delta_{x}(\varepsilon)$ and by the fact that $x_k\in A$ (hence $r(x_k)=x_k$). 
In each case we conclude that $\zeta(s)\in\mathcal{V}_{\varepsilon}$ for all $s\in\mathopen{[}0,1\mathclose{]}$.
The second property is obvious from the definition of $\zeta$. Thus our claim is proved.

From the two properties it follows that the path $\zeta$ is contained in the cylinder $\mathcal{C}$ and 
it has a nonempty intersection with the left and the right bases. Thus, taking $\sigma:=\zeta$, hypothesis $(ii)$ implies that there must be a sub-path
$\gamma$ of $\sigma$ such that $\gamma\subseteq W$ with $\phi(\gamma)\subseteq \mathcal{C}$, $\phi(\gamma)\cap \mathcal{C}_{l}\neq\emptyset$ and
$\phi(\gamma)\cap \mathcal{C}_{r}\neq\emptyset$. 
Let
\begin{equation*}
\xi=(\xi_{1},\xi_{2})\colon\mathopen{[}0,1\mathclose{]}\to\mathbb{R}\times X
\end{equation*}
be a continuous map such that the path $(\xi,\overline{\xi})$ is the reparametrization to $\mathopen{[}0,1\mathclose{]}$ 
of the sub-path $(\gamma,\overline{\gamma})$. By the assumptions, it follows that
\begin{itemize}
	\item $\xi(s)\in\mathcal{V}_{\varepsilon}\cap W$, for all $s\in\mathopen{[}0,1\mathclose{]}$;
	\item $\phi(\xi(s))\in \mathcal{C}$, for all $s\in\mathopen{[}0,1\mathclose{]}$;
	\item $\phi(\xi(0))\in \mathcal{C}_{l}$ and $\phi(\xi(1))\in \mathcal{C}_{r}$.
\end{itemize}

We consider now the continuous map $g\colon\mathopen{[}0,1\mathclose{]}\to \mathbb{R}\times X$ defined by $g(s)=\xi_{1}(s)-\phi_{1}(\xi(s))$. 
Since $\xi(s)\in \gamma\subseteq \sigma\subseteq \mathcal{C}$ for all $s\in\mathopen{[}0,1\mathclose{]}$,
we obtain that $\xi_{1}(0)\geq a$, so $g(0)\geq a-a=0$. Similarly $g(1)\leq0$. 
By Bolzano theorem, we derive that there exists $\hat{s}=\hat{s}_{\varepsilon}\in\mathopen{[}0,1\mathclose{]}$ 
such that $g(\hat{s})=\xi_{1}(\hat{s})-\phi_{1}(\xi(\hat{s}))=0$. Setting
\begin{equation*}
\hat{t}=\hat{t}_{\varepsilon}:=\xi_{1}(\hat{s}), \qquad \hat{x}=\hat{x}_{\varepsilon}:=\xi_{2}(\hat{s}), \qquad \hat{z}=\hat{z}_{\varepsilon}:=(\hat{t},\hat{x}),
\end{equation*}
it follows that
\begin{equation*}
\hat{z}=\xi(\hat{s})\in\mathcal{V}_{\varepsilon}\cap W, \qquad \hat{t}=\xi_{1}(\hat{s})=\phi_{1}(\hat{z}), 
\qquad \phi_{2}(\hat{z})=\phi_{2}(\xi(\hat{s}))\in A.
\end{equation*}

From the definition of $\psi$, we obtain that
\begin{equation}\label{eq1}
\hat{z}\in\mathcal{V}_{\varepsilon}\cap W, \qquad \hat{t}=\psi_{1}(\hat{z}), \qquad \phi_{2}(\hat{z})\in A.
\end{equation}
Furthermore, since $\hat{z}\in\mathcal{V}_{\varepsilon}$, there is an index $i\in\{1,\ldots,m\}$ such that $\|\hat{z}-(t_{i},x_{i})\|<\varepsilon$. 
Hence, we conclude that, for every $0 < \varepsilon < \text{dist}(\hat{A}, X\setminus \tilde{A})$, there exists a point $\hat{z}=\hat{z}_{\varepsilon}$ 
satisfying \eqref{eq1} and, moreover, there exists $\hat{y}_{\varepsilon}\in S$ such that $\|\hat{z}_{\varepsilon}-\hat{y}_{\varepsilon}\|<\varepsilon$.

Setting $\varepsilon=\varepsilon_{n}=1/n$, we derive that, for $n>1/\text{dist}(\hat{A}, X\setminus \tilde{A})$, 
there exists an element $\hat{z}_{n}\in\mathcal{V}_{1/n}\cap W$ with the characteristics listed above. 
In addition, for each $\hat{z}_{n}$ there exists a $\hat{y}_{n}\in S$ such that $\|\hat{z}_{n}-\hat{y}_{n}\|<1/n$. 
Since $S$ is compact, possibly passing to a subsequence, the sequence $(\hat{y}_{n})_{n}$ converges to an element
$\tilde{z}=(\tilde{t},\tilde{x})\in S$. 
Passing to the corresponding sub-sequence of $(\hat{z}_{n})_{n}$, we also obtain that $(\hat{z}_{n})_{n}$ converges to $\tilde{z}\in S$.
By the continuity of $\phi$ and $\psi$ and using the fact that $W$ and $A$ are closed, we have
\begin{equation*}
\tilde{z}\in S\cap W, \qquad \tilde{t}=\psi_{1}(\tilde{z}), \qquad \phi_{2}(\tilde{z})\in A.
\end{equation*}
Being $S$ contained in the set $\Sigma$, we find that
\begin{equation*}
\tilde{x}=\psi_{2}(\tilde{t},\tilde{x})
\end{equation*}
and consequently $\tilde{z}=(\tilde{t},\tilde{x})\in W$ is a fixed point of $\psi$. 
As shown at the beginning of the proof and since $\phi_{2}(\tilde{z})\in A$, 
we get that $\tilde{z}\in W$ is a fixed point of $\phi$. The theorem follows.
\end{proof}

\begin{remark}
We observe that in Theorem~\ref{ThPapiniZanolinretract} it is not restrictive to suppose that $\phi(D\cap\mathcal{C})\subseteq \mathbb{R}\times A$. 
In fact, if $\phi(D\cap\mathcal{C})\subseteq \mathbb{R}\times X$, we define 
\begin{equation*}
\hat{D}:=\{z\in D\cap\mathcal{C}\colon\phi_{2}(z)\in A\}\subseteq\mathcal{C}
\end{equation*}
and we consider the restriction $\hat{\phi}:=\phi|_{\hat{D}}\colon\hat{D}\to \mathbb{R}\times A$.
Now we prove that $\hat{\phi}$ satisfies the two conditions of the theorem with respect to a suitable closed set $\hat{W}\subseteq\hat{D}$. Setting
\begin{equation*}
\hat{W}:=\{z\in W\colon\phi(z)\in\mathcal{C}\}\subseteq\hat{D},
\end{equation*}
clearly condition $(i)$ is satisfied (as $\hat{W}\subseteq W$). We prove $(ii)$. 
Suppose to have a path $\sigma\subseteq \mathcal{C}$ with $\sigma\cap \mathcal{C}_{l}\neq\emptyset$ and $\sigma\cap \mathcal{C}_{r}\neq\emptyset$.
We know that there exists a sub-path $\gamma\subseteq\sigma\cap W$ with $\phi(\gamma)\subseteq \mathcal{C}$, 
$\phi(\gamma)\cap \mathcal{C}_{l}\neq\emptyset$ and $\phi(\gamma)\cap \mathcal{C}_{r}\neq\emptyset$. Obviously $\gamma\subseteq\hat{W}$. Hence $(ii)$ follows.
\end{remark}

\begin{remark}
We note that in condition $(ii)$ of Theorem~\ref{ThPapiniZanolinretract} we can replace the term \textit{path} 
with the concept of \textit{continuum} (i.e.~compact and connected set). 
Following the proof of \cite[Theorem~7]{PaZa-2007}, we obtain the same thesis in the case of retracts.
\end{remark}

Our next aim is to give an existence result with respect to a cylinder having an absolute retract as a ``base''. For reader's convenience we recall
that a space $A$ is an \textit{absolute retract} (or simply an \textit{AR}) if $A$ is metrizable and for any metrizable $Y$ and every closed
$M\subseteq Y$ each continuous function $f\colon M\to A$ is continuously extendable over $Y$. We also recall that a space $A$ is an \textit{absolute
neighbourhood retract} (or simply an \textit{ANR}) if $A$ is metrizable and for any metrizable $Y$, every closed $M\subseteq Y$ and each continuous
function $f\colon M\to A$ there exists a neighbourhood $U$ of $M$ and a continuous extension of $f$ over $U$. Clearly an AR is an ANR.
For properties and characterization of ARs and ANRs, we refer to \cite{GrDu-2003,Hu-1965}.

The next corollary deals with the case of a cylinder with base an AR.

\begin{corollary}
\label{ThPapiniZanolinAR}
Let $A$ be an AR and $a,b\in\mathbb{R}$ with $a<b$. We denote by
\begin{equation*}
\mathcal{C}:=\mathopen{[}a,b\mathclose{]}\times A
\end{equation*}
the cylinder with ``base'' the absolute retract $A$ and ``height'' the interval $\mathopen{[}a,b\mathclose{]}$.
We indicate with
\begin{equation*}
\mathcal{C}_{l}:=\{a\}\times A \quad \text{ and } \quad \mathcal{C}_{r}:=\{b\}\times A
\end{equation*}
the cylinder's left base and the cylinder's right base respectively.

Let $D\subseteq\mathcal{C}$. Consider the operator
\begin{equation*}
\phi=(\phi_{1},\phi_{2})\colon D\to\mathbb{R}\times A
\end{equation*}
and suppose that there exists a closed set $W\subseteq D$ such that
\begin{enumerate}
\item [$(i)$] $\phi$ is compact on $W$;
\item [$(ii)$] for every path $\sigma\subseteq \mathcal{C}$ with $\sigma\cap \mathcal{C}_{l}\neq\emptyset$ 
and $\sigma\cap \mathcal{C}_{r}\neq\emptyset$ there exists a sub-path $\gamma\subseteq\sigma\cap W$
with $\phi(\gamma)\subseteq \mathcal{C}$, $\phi(\gamma)\cap \mathcal{C}_{l}\neq\emptyset$ and $\phi(\gamma)\cap \mathcal{C}_{r}\neq\emptyset$.
\end{enumerate}
Then there exists $\tilde{z}\in W$ such that $\phi(\tilde{z})=\tilde{z}$.
\end{corollary}

\begin{proof}
Given the absolute retract $A$, we can fix a normed linear space $X$ where $A$ is a retract (cf.~\cite[Theorem~7.6, p.~164]{GrDu-2003}).
Finally we use Theorem~\ref{ThPapiniZanolinretract} to conclude.
\end{proof}

The following example shows that completely continuity is not sufficient for the claim of Theorem~\ref{ThPapiniZanolinretract}. We recall that a map 
$f\colon X\to Y$ between normed linear spaces is \textit{completely continuous}, if $f$ is continuous 
and the closure of $f(B)$ is a compact subset of $Y$, for each bounded subset $B\subseteq X$.

\begin{example}
Consider the map $\phi\colon\mathopen{[}0,1\mathclose{]}\times \mathbb{R}\to\mathopen{[}0,1\mathclose{]}\times \mathbb{R}$, $\phi(t,x)=(t,x+1)$, 
and set $W=\mathopen{[}0,1\mathclose{]}\times\mathbb{R}$. Clearly $\mathbb{R}$ is a retract (of $\mathbb{R}$) and $W$ is closed. 
We note that the map $\phi$ is completely continuous, but it is not compact on $W$, since $\phi(W)=\mathopen{[}0,1\mathclose{]}\times\mathbb{R}$.
Moreover, condition $(ii)$ is satisfied, but $\phi$ has not fixed points.
\end{example}

With the following example we point out that, if we have an ANR at the place of the AR, the assertion under the same hypotheses is not true.

\begin{example}
Let $S^{n-1}:=\{x\in\mathbb{R}^{n}\colon\|x\|=1\}$ be the unit sphere in $\mathbb{R}^{n}$. 
It is known that $S^{n-1}$ is an ANR but not an AR (see \cite[Theorem~1.5, p.~280]{GrDu-2003} and \cite[p.~20--21]{Hu-1965}).
Fix $\alpha\in\mathopen{]}0,2\pi\mathclose{[}$ ($\alpha/2\pi\notin\mathbb{Q}$) and consider the map 
$\phi\colon\mathopen{[}0,1\mathclose{]}\times S^{1}\to\mathopen{[}0,1\mathclose{]}\times S^{1}$, 
$\phi(t,x)=(t,xe^{i\alpha})$, which represent the rotation around zero of angle $\alpha$. 
Clearly $\phi$ is continuous (hence compact), but it has not fixed points (or periodic points).
\end{example}

In almost all the applications to differential equations we can not work with a cylinder. 
Thus, our new goal is to present an existence result for fixed points of maps defined in more general domains. 

We recall that a topological space $X$ satisfies the \textit{fixed point property} if any continuous function $f\colon X\to X$ has a fixed point.
From classical results, we know that an homeomorphism preserves the fixed point property. In the following result we show that also our property is
preserved under homeomorphisms.

\begin{corollary}
\label{ThPapiniZanolinhomeomorphism}
Let $Y$ be a metric space and $M\subseteq Y$ a closed subset homeomorphic to the cylinder $\mathcal{C}:=\mathopen{[}a,b\mathclose{]}\times A$, 
where $A$ is an AR and $a,b\in\mathbb{R}$ with $a<b$. Call
\begin{equation*}
h\colon M\to\mathopen{[}a,b\mathclose{]}\times A
\end{equation*}
the homeomorphism onto its image and
\begin{equation*}
M_{l}:=h^{-1}(\{a\}\times A) \quad \text{ and } \quad M_{r}:=h^{-1}(\{b\}\times A)
\end{equation*}
the ``left base'' and the ``right base'' of $M$ respectively.

Let $N\subseteq Y$. Consider the operator $\phi\colon N\to Y$ and suppose that there exists a closed set $V\subseteq M\cap N$ such that
\begin{enumerate}
	\item [$(i)$] $\phi$ is compact on $V$;
	\item [$(ii)$] for every path $\sigma\subseteq M$ with $\sigma\cap M_{l}\neq\emptyset$ and $\sigma\cap M_{r}\neq\emptyset$ 
there exists a sub-path $\gamma\subseteq\sigma\cap V$ with $\phi(\gamma)\subseteq M$, 
$\phi(\gamma)\cap M_{l}\neq\emptyset$ and $\phi(\gamma)\cap M_{r}\neq\emptyset$.
\end{enumerate}
Then there exists $\tilde{z}\in V$ such that $\phi(\tilde{z})=\tilde{z}$.
\end{corollary}

\begin{proof}
Define
\begin{equation*}
H:=\{z\in V\colon\phi(z)\in M\}
\end{equation*}
and observe that $H=\phi^{-1}(M)\cap V$ is closed (since $M$ and $V$ are closed), $\hat{\phi}:=\phi|_{H}\colon H\to Y$ 
is compact and $\phi(H)\subseteq M$. Set $D:=h(H)$, which is closed and contained in $h(M)=\mathcal{C}$.

Consider this diagram
\begin{displaymath}
\xymatrix{
&H(\subseteq V)
\ar[rr]^{\mathlarger{\hat{\phi}}}
&&M
\ar[d]^{\mathlarger{\, h}}
\\
&h(H)=D
\ar@{-->}[rr]_{\mathlarger{\hat{\psi}}}
\ar[u]^{\mathlarger{h^{-1}}}
&&h(M)=\mathcal{C}
}
\end{displaymath}
where
\begin{equation*}
\hat{\psi}:= h\circ\hat{\phi}\circ h^{-1}\colon D\to\mathcal{C}=\mathopen{[}a,b\mathclose{]}\times A.
\end{equation*}
We claim that $\hat{\psi}$ is a compact map. Suppose to have a sequence $(y_{n})_{n}$ contained in $\hat{\psi}(D)$, 
i.e.~$y_{n}=\hat{\psi}(x_{n})=h(\hat{\phi}(h^{-1}(x_{n})))$ with $x_{n}\in D$. By the compactness of $\hat{\phi}$, 
the closure of $\hat{\phi}(h^{-1}(D))=\hat{\phi}(H)$ is compact. Then there exists a subsequence $(x_{k_{n}})$ 
such that $\hat{\phi}(h^{-1}(x_{k_{n}}))\to \bar{x}\in M$ (since $M$ closed). 
By the continuity of $h$, we conclude that $y_{k_{n}}=\hat{\psi}(x_{k_{n}})=h(\hat{\phi}(h^{-1}(x_{k_{n}})))\to h(\bar{x})\in h(M)=\mathcal{C}$. 
Hence the compactness of $\hat{\psi}$.

Consider a path $\hat{\sigma}\subseteq \mathcal{C}$ with $\hat{\sigma}\cap \mathcal{C}_{l}\neq\emptyset$ 
and $\hat{\sigma}\cap \mathcal{C}_{r}\neq\emptyset$. Consequently $\sigma:=h^{-1}(\hat{\sigma})\subseteq M$, $\sigma\cap M_{l}\neq\emptyset$ 
and $\sigma\cap M_{r}\neq\emptyset$. Then there exists a sub-path $\gamma\subseteq\sigma\cap V$ with $\phi(\gamma)\subseteq M$, 
$\phi(\gamma)\cap M_{l}\neq\emptyset$ and $\phi(\gamma)\cap M_{r}\neq\emptyset$. Hence $\hat{\gamma}:=h(\gamma)$ is a sub-path of $\hat{\sigma}$
such that $\hat{\gamma}\subseteq\hat{\sigma} \cap D$, $\hat{\psi}(\hat{\gamma})\subseteq \mathcal{C}$, $\hat{\psi}(\hat{\gamma})\cap \mathcal{C}_{l}\neq\emptyset$ 
and $\hat{\psi}(\hat{\gamma})\cap \mathcal{C}_{r}\neq\emptyset$.

Setting $W=D$, every hypothesis of Corollary~\ref{ThPapiniZanolinAR} is satisfied. Then there exists $\bar{z}\in D=h(H)$ 
such that $\hat{\psi}(\bar{z})=\bar{z}$. We conclude that the point $\tilde{z}=h^{-1}(\bar{z})\in V$ is a fixed point of $\phi$.
\end{proof}

\section{Some classical theorems}\label{section3}

In this section we present three classical results of fixed point theory revised in view of Theorem~\ref{ThPapiniZanolinretract}. 
In other words, we check that the main theorem is a generalization of these results.

Certainly \textit{Brouwer fixed point theorem} is one of the most famous. In the proposed proof we note that
 a simple trick allows us to put ourselves in the hypotheses $(i)$ and $(ii)$ of Theorem~\ref{ThPapiniZanolinretract}.

\begin{theorem}[Brouwer fixed point theorem]
\label{Brouwer fixed point theorem}
Let $B=B[0,1]\subseteq \mathbb{R}^{n}$ be the unit closed ball. Any continuous function $f\colon B\to B$ has a fixed point in $B$.
\end{theorem}

\begin{proof}
First of all we observe that $B$ is a retract of $\mathbb{R}^{n}$ and $f$ is compact (by the finite dimension of the space). 
Consider the application $\phi=(\phi_{1},\phi_{2})\colon\mathopen{[}0,1\mathclose{]}\times B\to\mathopen{[}0,1\mathclose{]}\times B$ defined as
\begin{equation*}
\phi_{1}(t,x):=t, \qquad \phi_{2}(t,x):=f(x).
\end{equation*}
By the hypothesis, $\phi$ is well defined and satisfies conditions $(i)$ and $(ii)$ of Theorem~\ref{ThPapiniZanolinretract} (with
$W=\mathopen{[}0,1\mathclose{]}\times B$). Then there exists $\tilde{z}=(\tilde{t},\tilde{x})\in \mathopen{[}0,1\mathclose{]}\times B$ such that
$\phi(\tilde{z})=\tilde{z}$, i.e.~$(\tilde{t},\tilde{x})=(\tilde{t},f(\tilde{x}))$. We conclude that $\tilde{x}\in B$ is a fixed point of $f$.
\end{proof}

The second classical result we look over is \textit{Schauder fixed point theorem}.

\begin{theorem}[Schauder fixed point theorem]
\label{Schauder fixed point theorem}
Let $C$ be a closed, convex and nonempty subset of a normed linear space $X$. Any compact map $\varphi\colon C\to C$ has a fixed point in $C$.
\end{theorem}

Now we prove directly a more general theorem, of which one can find an alternative proof in \cite[Theorem~7.9, p.~165]{GrDu-2003}.

\begin{theorem}[Generalized Schauder theorem]
\label{generalized Schauder fixed point theorem}
Let $Z$ be an AR. Any compact map $\varphi\colon Z\to Z$ has a fixed point in $Z$.
\end{theorem}

\begin{proof}
Consider the application $\phi=(\phi_{1},\phi_{2})\colon \mathopen{[}0,1\mathclose{]}\times Z\to\mathopen{[}0,1\mathclose{]}\times Z$ defined as
\begin{equation*}
\phi_{1}(t,x):=t, \qquad \phi_{2}(t,x):=\varphi(x).
\end{equation*}
By the hypothesis, $\phi$ is well defined and satisfies conditions $(i)$ and $(ii)$ of Corollary~\ref{ThPapiniZanolinAR} (with
$W=\mathopen{[}0,1\mathclose{]}\times Z$). Then there exists $\tilde{z}=(\tilde{t},\tilde{x})\in \mathopen{[}0,1\mathclose{]}\times Z$ such that
$\phi(\tilde{z})=\tilde{z}$, i.e.~$(\tilde{t},\tilde{x})=(\tilde{t},\varphi(\tilde{x}))$. We conclude that $\tilde{x}\in Z$ is a fixed point of
$\varphi$.
\end{proof}

It is well known that Brouwer fixed point theorem is equivalent to the classical \textit{Poin\-ca\-r\'e-Miranda zeros theorem}. 
Now we present a theorem that is the ``fixed point version'' of Poincar\'e-Miranda result, alternatively it can be viewed as a $n$-dimensional
generalization of Bolzano theorem. Our aim is to provide a direct proof 
in the light of the main theorem, without using previous classical results.

The theorem we state below has been proposed in some recent works, in this or in other versions.
As examples, see~\cite{FoGi-2015pp,Ka-2000,Kw-2008fpta,Kw-2008,Mar-2006,Ma-2013ans,Ma-2013,PiZa-2007}.

\begin{theorem}[Poincar\'e-Miranda]
\label{Poincare-Miranda fixed point theorem}
Let $\mathcal{R}=\prod_{i=1}^{n}\mathopen{[}a_{i},b_{i}\mathclose{]}$ be a $n$-dimensional rectangle. We denote by
$I_{i}^{-}=\{x\in\mathcal{R}\colon x_{i}=a_{i}\}$ and $I_{i}^{+}=\{x\in\mathcal{R}\colon x_{i}=b_{i}\}$ its $i$-faces. 
Let $g=(g_{1},\ldots,g_{n})\colon \mathcal{R}\to \mathbb{R}^{n}$ be a continuous vector field such that for each $i\in\{1,\ldots,n\}$ 
one of the two following possibilities is true:
\begin{itemize}
 \item [$(e_{i})$] \quad $g_{i}(x)\leq a_{i}$, $\forall \, x\in I_{i}^{-}$, \quad  $g_{i}(x)\geq b_{i}$, $\forall \, x\in I_{i}^{+}$;
 \item [$(c_{i})$] \quad $g_{i}(x)\geq a_{i}$, $\forall \, x\in I_{i}^{-}$, \quad  $g_{i}(x)\leq b_{i}$, $\forall \, x\in I_{i}^{+}$.
\end{itemize}
Then there exists $z\in\mathcal{R}$ such that $g(z)=z$.
\end{theorem}

\begin{proof}
First of all we introduce a useful notation. 
Let $I:=\{1,\ldots,n\}$, $I_{e}:=\{i\in I\colon (e_{i})\text{ is valid}\}$ 
and $I_{c}:=\{i\in I\colon (c_{i})\text{ is valid}\}$. By the hypotheses, $I=I_{e}\cup I_{c}$ and $I_{e}\cap I_{c}=\emptyset$.

We first assume that $\emptyset\neq I_{e} =\{k\}$. We define
\begin{equation*}
\mathcal{R}_{k}:=\mathopen{[}a_{1},b_{1}\mathclose{]}\times\dots\times\mathopen{[}a_{k-1},b_{k-1}\mathclose{]}
\times\mathopen{[}a_{k+1},b_{k+1}\mathclose{]}\times\dots\times\mathopen{[}a_{n},b_{n}\mathclose{]}
\end{equation*}
and we observe that $\mathcal{R}_{k}$ is a retract of $\mathbb{R}^{n-1}$, hence $\mathopen{[}a_k,b_k\mathclose{]}\times\mathcal{R}_{k}$ 
is a cylinder where we can apply Theorem~\ref{ThPapiniZanolinretract}. Without loss of generality and for simplicity, we assume $k=1$. 
Consider the operator 
\begin{equation*}
\tilde{g}\colon \mathopen{[}a_{1},b_{1}\mathclose{]}\times\mathcal{R}_{1}=\mathcal{R}\to \mathbb{R}^{n} = \mathbb{R}\times\mathbb{R}^{n-1}
\end{equation*}
defined as
\begin{equation*}
\tilde{g}(x):=(g_{1}(x),p_{2}(g_{2}(x)),\ldots,p_{n}(g_{n}(x))),
\end{equation*}
where $p_{i}(x):=\min\{b_{i},\max\{x_{i},a_{i}\}\}$ is the projection on $\mathopen{[}a_{i},b_{i}\mathclose{]}$, for all $i\in I_{c}$. 
We claim that $\tilde{g}$ and $W=\mathcal{R}$ satisfy conditions $(i)$ and $(ii)$ of Theorem~\ref{ThPapiniZanolinretract}. Clearly $\tilde{g}$ is
continuous (hence compact). Suppose that the path $(\sigma,\overline{\sigma})$, with $\sigma\colon \mathopen{[}0,1\mathclose{]}\to\mathcal{R}$, is
such that $\sigma(0)\in I_{1}^{-}$ and $\sigma(1)\in I_{1}^{+}$. Hypothesis $(e_{1})$ implies that $g_{1}(\sigma(0))\leq a_{1}$ and $g_{1}(\sigma(1))\geq
b_{1}$. Then, by the continuity of $g_{1}$ and $\sigma$, there exists a sub-interval $\mathopen{[}s_{0},s_{1}\mathclose{]}\subseteq\mathopen{[}0,1\mathclose{]}$ such that
$g_{1}(\sigma(\mathopen{[}s_{0},s_{1}\mathclose{]}))=\mathopen{[}a_{1},b_{1}\mathclose{]}$. In particular the sub-path
$\gamma:=\sigma|_{\mathopen{[}s_{0},s_{1}\mathclose{]}}$ is such that $\tilde{g}(\gamma)\subseteq\mathcal{R}$, $\tilde{g}(\gamma)\cap
I_{1}^{-}\neq\emptyset$ and $\tilde{g}(\gamma)\cap I_{1}^{+}\neq\emptyset$. We deduce the validity of $(ii)$.

By Theorem~\ref{ThPapiniZanolinretract}, there exists an element $z=(z_{1},z_{2},\ldots,z_{n})\in\mathcal{R}$ such that $\tilde{g}(z)=z$. 
We claim that $z\in\mathcal{R}$ is a fixed point of $g$. We have to prove that $g_{i}(z)\in\mathopen{[}a_{i},b_{i}\mathclose{]}$ for all $i\in I_{c}$. 
Suppose that there is an index $i\in I_{c}$ such that $g_{i}(z)<a_{i}$ or $g_{i}(z)>b_{i}$. Consider the first case. 
Then $z_{i}=p_{i}(g_{i}(z))=a_{i}$, so $g_{i}(z)\geq a_{i}$, by $(c_{i})$. This is a contradiction. 
Similarly one can prove that $g_{i}(z)>b_{i}$ does not occur. We have the assertion.

Now suppose that $\emptyset\neq I_{e}=\{1,\ldots,j\}$. Define $\hat{g}\colon\mathcal{R}\to \mathbb{R}^{n}$ as follows
\begin{equation*}
\hat{g}(x):=(g_{1}(x),2x_{2}-g_{2}(x),\ldots,2x_j-g_j(x),g_{j+1}(x),\ldots,g_{n}(x)).
\end{equation*}
It is easy to see that $\hat{g}$ satisfies $(e_{1})$ and also $(c_{i})$ for all $i\in\{2,\ldots,n\}$. 
Consequently we can apply the first case and observe that a fixed point of $\hat{g}$ is also a fixed point of $g$.

The last case is $I_{e}=\emptyset$. We proceed as in the second case, defining $\hat{g}\colon\mathcal{R}\to \mathbb{R}^{n}$ as
\begin{equation*}
\hat{g}(x):=(2x_{1}-g_{1}(x),g_{2}(x),\ldots,g_{n}(x)).
\end{equation*}
\end{proof}

\section{Other consequences of the main theorem}\label{section4}

We conclude our work by explore other fixed point theorems available in the literature. 
The final aim is to relate Krasnosel'ski\u{\i} theorems on cones to our main result.

First of all we analyze still a fixed point theorem on cylinders: \textit{Kwong's theorem}. 
We state and prove a more general version compared to \cite[Theorem~3.2]{Kw-2008fpta} and
\cite[Theorem~3]{Kw-2008}.

\begin{theorem}\label{Kwong fixed point theorem}
Let $A$ be an AR and $\mathcal{C}$, $\mathcal{C}_{l}$, $\mathcal{C}_{r}$ as in Corollary~\ref{ThPapiniZanolinAR}.
Consider a compact map $T=(T_{1},T_{2})\colon\mathcal{C}\to \mathbb{R}\times A$.
\begin{itemize}
	\item (\textit{Expansive form}) $T$ has at least a fixed point in $\mathcal{C}$ if
					\begin{itemize}
						\item [$(E_{a})$] $T_{1}(z)\leq a$, $\forall \, z\in\mathcal{C}_{l}$;
						\item [$(E_{b})$] $T_{1}(z)\geq b$, $\forall \, z\in\mathcal{C}_{r}$.
					\end{itemize}
	\item (\textit{Compressive form}) $T$ has at least a fixed point in $\mathcal{C}$ if
					\begin{itemize}
						\item [$(C_{a})$] $T_{1}(z)\geq a$, $\forall \, z\in\mathcal{C}_{l}$;
						\item [$(C_{b})$] $T_{1}(z)\leq b$, $\forall \, z\in\mathcal{C}_{r}$.
					\end{itemize}
\end{itemize}
\end{theorem}

\begin{proof}
\textit{Expansive form}. We want to prove that conditions $(i)$ and $(ii)$ of Corollary~\ref{ThPapiniZanolinAR} hold,
with respect to $W=\mathcal{C}$. By the compactness hypothesis, $(i)$ follows. Now suppose that the path $(\sigma,\overline{\sigma})$, with
$\sigma\colon\mathopen{[}0,1\mathclose{]}\to\mathcal{C}$, is such that $\sigma(0)\in\mathcal{C}_{l}$ and $\sigma(1)\in\mathcal{C}_{r}$. 
By hypotheses $(E_{a})$ and $(E_{b})$, $T_{1}(\sigma(0))\leq a$ and $T_{1}(\sigma(1))\geq b$. Then, by the continuity of $T_{1}$ and $\sigma$,
there exists a sub-interval $\mathopen{[}s_{0},s_{1}\mathclose{]}\subseteq\mathopen{[}0,1\mathclose{]}$ such that
$T_{1}(\sigma(\mathopen{[}s_{0},s_{1}\mathclose{]}))=\mathopen{[}a,b\mathclose{]}$. 
In particular the sub-path $\gamma:=\sigma|_{\mathopen{[}s_{0},s_{1}\mathclose{]}}$ is such that $T(\gamma)\subseteq \mathcal{C}$, $T(\gamma)\cap
\mathcal{C}_{l}\neq\emptyset$ and $T(\gamma)\cap \mathcal{C}_{r}\neq\emptyset$. We deduce the validity of $(ii)$.
By Corollary~\ref{ThPapiniZanolinAR} we have the assertion.

\smallskip

\noindent
\textit{Compressive form}. We reduce this case to the expansive form. Define a new operator $S=(S_{1},S_{2})\colon\mathcal{C}\to \mathbb{R}\times A$ as follows
\begin{equation*}
S_{1}(t,x):=2t-T_{1}(t,x), \qquad S_{2}(t,x):=T_{2}(t,x).
\end{equation*}
We observe that the operator $S$ is compact, since $T$ is compact. Moreover, if $t=a$, $T_{1}(t,x)\geq a$ by $(C_{a})$, then $S_{1}(t,x)\leq2a-a=a$. 
Similarly $S_{2}(t,x)\geq b$ if $t=b$, using $(C_{b})$. Then $S$ satisfies the expansive form conditions. 
Consequently there is $\tilde{z}=(\tilde{t},\tilde{x})\in\mathcal{C}$ such that $S(\tilde{z})=\tilde{z}$. Hence
\begin{equation*}
\tilde{t}=S_{1}(\tilde{t},\tilde{x})=2\tilde{t}-T_{1}(\tilde{t},\tilde{x}), \qquad \tilde{x}=S_{2}(\tilde{t},\tilde{x})=T_{2}(\tilde{t},\tilde{x}).
\end{equation*}
We conclude that $\tilde{z}\in\mathcal{C}$ is also a fixed point of $T$.
\end{proof}

\begin{remark}
In the second part of the proof of Theorem~\ref{Kwong fixed point theorem} we bring back the compressive form to the expansive framework
employing the same trick that Kwong uses in \cite{Kw-2008fpta} to transform an expansive operator into a compressive one.

An alternative way to prove the compressive form is to use directly the generalized Schauder fixed point theorem that, 
as already noted, is a direct consequence of our main result.
\end{remark}

Till now we have always considered fixed point theorems of expansive type or we have reduced the problem to them. 
In view of Corollary~\ref{ThPapiniZanolinhomeomorphism}, we present a result of compressive type that allows us to easily prove Krasnosel'ski\u{\i} fixed point theorems.

\begin{corollary}
\label{ThCompressiveHomeomorphism}
Let $Y$ be a metric space and $M'\subseteq Y$ a closed subset homeomorphic to the cylinder $\mathcal{C'}:=\mathopen{[}a',b'\mathclose{]}\times A$,
where $A$ is an AR and $a',b'\in\mathbb{R}$ with $a'<b'$. Call $h\colon M'\to\mathcal{C'}$ the homeomorphism onto its image. Let $M\subseteq M'$ and suppose
there exist $a,b\in\mathbb{R}$ with $a'<a<b<b'$ such that $h(M)=\mathcal{C}:=\mathopen{[}a,b\mathclose{]}\times A$. We define
\begin{equation*}
M_{l}:=h^{-1}(\{a\}\times A) \quad \text{ and } \quad M_{r}:=h^{-1}(\{b\}\times A)
\end{equation*}
the left base and the right base of $M$ respectively.

Consider the operator $\phi\colon M\to Y$ and suppose that
\begin{enumerate}
	\item [$(i)$] $\phi$ is compact on $M$;
	\item [$(ii)$] 	$\phi(M_l)\subseteq h^{-1}(\mathopen{[}a,b'\mathclose{]}\times A)$ and 
								$\phi(M_r)\subseteq h^{-1}(\mathopen{[}a',b\mathclose{]}\times A)$.
\end{enumerate}
Then there exists $\tilde{z}\in M$ such that $\phi(\tilde{z})=\tilde{z}$.
\end{corollary}

\begin{proof}
First of all we observe that condition $(ii)$ is equivalent to require that $\psi(\{a\}\times A)\subseteq \mathopen{[}a,b'\mathclose{]}\times A$
and $\psi(\{b\}\times A)\subseteq\mathopen{[}a',b\mathclose{]}\times A$, where $\psi:=h\circ\phi\circ h^{-1}$.
Secondly, as in the proof of Corollary~\ref{ThPapiniZanolinhomeomorphism}, we consider the homeomorphism $h|_{M}$ between $M$ and $\mathcal{C}$. 
Finally, we apply the compressive form in Theorem~\ref{Kwong fixed point theorem} to $T=\psi$ and $\mathcal{C}$. Hence, we obtain the assertion.
\end{proof}

\begin{remark}
As we shall see in the sequel, Corollary~\ref{ThCompressiveHomeomorphism} is a useful tool to prove compressive forms. 
However, we can also note that, unlike Corollary~\ref{ThPapiniZanolinhomeomorphism}, 
we do not have a localization of fixed points (for example inside a closed subset of $M$).
\end{remark}

Hereafter we deal with cones in normed linear spaces. 
Let $(X,\|\cdot\|)$ be a normed linear space. We recall that 
a subset $K\subseteq X$ is a \textit{cone} if the following conditions are satisfied:
\begin{itemize}
 \item $K$ is closed;
 \item $\alpha u + \beta v \in K$, $\forall \, u,v\in K$, $\forall \, \alpha,\beta \in \mathbb{R}^{+}$;
 \item $K\cap(-K)=\{0\}$, i.e.~if $u\in K$ and $-u\in K$ then $u=0$.
\end{itemize}

In the following example we analyze some subsets of a cone in a normed space. 
What we observe allows us to define a homeomorphism between a region inside the cone and a suitable cylinder.

\begin{example}\label{exampleK}
Let $(X,\|\cdot\|)$ be a normed linear space and $K\subseteq X$ a cone.
By Dugundji extension theorem, $K$ is an AR, since $K$ is convex.

Consider the subset $K_{a}:=\{x\in K\colon \|x\|=a\}$, with $a\in\mathbb{R}^{+}$. 
We claim that $K_{a}$ is a retract of $K$, therefore it is an AR (cf.~\cite[Proposition~7.2, p.~162]{GrDu-2003}). If $a=0$, we consider the retraction
$r_{a}\equiv0$. If $a\neq0$, we fix an element $y\in K\setminus\{0\}$ and we define the retraction $r_{a}\colon K\to K_{a}$ as follows
\begin{equation*}
r_{a}(x):=a\,\dfrac{x+(a-\|x\|)^{2}y}{\|x+(a-\|x\|)^{2}y\|}, \quad x\in K.
\end{equation*}
The map $r_{a}$ is well defined, because $x+(a-\|x\|)^{2}y\in K\setminus\{0\}$, $\forall \, x\in K$. 
Moreover, $r_{a}(x)\in K$ and $\|r_{a}(x)\|=a$, $\forall \, x\in K$. Clearly $r_{a}$ is continuous and $r|_{K_{a}}=Id_{K_{a}}$. 

Let $l\colon X\to\mathbb{R}$ be a continuous functional, positive (i.e.~$l(x)\geq0$, $\forall \, x\in K$) and positively homogeneous. 
Consider the set $K_{a}^{l}:=\{x\in K\colon l(x)=a\}$, with $a\in\mathbb{R}^{+}$.
If the functional $l$ is also strictly positive (i.e.~$l(x)>0$, $\forall \, x\in K\setminus\{0\}$), 
using the previous retractions (where we replace $\|\cdot\|$ with $l(\cdot)$), we obtain that $K_{a}^{l}$ is a retract of $K$, therefore an AR.

Suppose that $l$ is not strictly positive and assume that $l$ is linear and $l\not\equiv0$ on $K$ (otherwise $K_{a}^{l}$ is the whole cone $K$ or the empty set). 
Fix an element $y\in K$ with $l(y)\neq0$. Fix $a\neq0$. Define the retraction $r_{a}\colon K\to K_{a}^{l}$ as follows
\begin{equation*}
r_{a}(x):=a\,\dfrac{x+(a-l(x))^{2}y}{l(x+(a-l(x))^{2}y)}, \quad x\in K.
\end{equation*}
The map $r_{a}$ is well defined, because $l(x+(a-l(x))^{2}y)=l(x)+(a-l(x))^{2}l(y)>0$, $\forall \, x\in K$. 
Moreover, $r_{a}(x)\in K$ and $l(r_{a}(x))=a$, $\forall \, x\in K$. Clearly $r_{a}$ is continuous and $r_{a}|_{K_{a}^{l}}=Id_{K_{a}^{l}}$. If $a=0$,
$K_{a}^{l}$ is closed, convex (by the linearity of $l$) and nonempty ($0\in K_{a}^{l}$). By Dugundji extension theorem, we conclude that $K_{0}^{l}$
is a retract of $X$, so it is an AR.

We note that we can alternatively suppose that $l\colon K\to\mathbb{R}^{+}$ is defined only on $K$, is positively homogeneous and is such that
\begin{equation*}
l(\alpha x_{1} + \beta x_{2}) \geq \alpha l(x_{1}) + \beta l(x_{2}), \quad  \forall \, x_{1},x_{2}\in K, \; \forall \, \alpha,\beta\geq0.
\end{equation*}
In the case of positively homogeneous functionals, we stress that this hypothesis is equivalent to assume that $l$ is concave on $K$.

Furthermore, we also remark that the assumption of positive homogeneity is necessary for our definition of the retractions $r_{a}$.

For definitions, properties and characterizations of functionals defined on cones, we refer to \cite{GuLa-1988}.
\end{example}

Now we have all the tools to enunciate and prove Krasnosel'ski\u{\i} fixed point theorem. 
The original version and the original proof are located in \cite{Kr-1960,Kr-1964}. 
Alternative proofs can be found in \cite{GrDu-2003,GuLa-1988,Sc-1976}.

\begin{theorem}[Krasnosel'ski\u{\i}]
\label{Krasnosel'skii normed fixed point theorem}
Let $(X,\|\cdot\|)$ be a normed linear space and $K$ a cone in $X$. 
Fix $a,b\in\mathbb{R}^{+}$ with $0<a<b$. Define $K_{a,b}:=\{x\in K\colon a\leq\|x\|\leq b\}$ 
and $K_{d}:=\{x\in K\colon \|x\|=d\}$, for all $d\in\mathbb{R}^{+}$.

Consider a compact map $T\colon K\to K$.
\begin{itemize}
 \item (\textit{Expansive form}) $T$ has at least a fixed point in $K_{a,b}$ if
				\begin{itemize}
					\item [$(E_{a})$] $\|T(x)\|\leq\|x\|$, $\forall \, x\in K_{a}$;
					\item [$(E_{b})$] $\|T(x)\|\geq\|x\|$, $\forall \, x\in K_{b}$.
				\end{itemize}
 \item (\textit{Compressive form}) $T$ has at least a fixed point in $K_{a,b}$ if
				\begin{itemize}
					\item [$(C_{a})$] $\|T(x)\|\geq\|x\|$, $\forall \, x\in K_{a}$;
					\item [$(C_{b})$] $\|T(x)\|\leq\|x\|$, $\forall \, x\in K_{b}$.
				\end{itemize}
\end{itemize}
\end{theorem}

\begin{proof}
Consider the cylinder $\mathcal{C}=\mathopen{[}a,b\mathclose{]}\times K_{1}$. 
As observed in Example \ref{exampleK}, the set $K_{1}$ is a retract of $K$ (hence of $X$). 
Now we define an homeomorphism $h$ between $K_{a,b}$ and $\mathcal{C}$ as follows
\begin{equation*}
h(x):=\biggl{(}\|x\|,\dfrac{x}{\|x\|}\biggr{)}, \quad x\in K_{a,b}.
\end{equation*}
Clearly $h$ is a continuous bijection with continuous inverse $h^{-1}\colon (\alpha,z)\mapsto \alpha z$.
Moreover, we note that $K_{a}=h^{-1}(\{a\}\times K_{1})$ and $K_{b}=h^{-1}(\{b\}\times K_{1})$.

\smallskip

\noindent
\textit{Expansive form}. 
We use Corollary~\ref{ThPapiniZanolinhomeomorphism} with $Y=X$, $N=K$ and $M=V=K_{a,b}$. 
Condition $(i)$ is obvious. We prove $(ii)$. Suppose that the path $(\sigma,\overline{\sigma})$, with $\sigma\colon\mathopen{[}0,1\mathclose{]}\to
K_{a,b}$, is such that $\sigma(0)\in K_{a}$ and $\sigma(1)\in K_{b}$. 
From $(E_{a})$ and $(E_{b})$ we derive that $\|T(\sigma(0))\|\leq a$ and
$\|T(\sigma(1))\|\geq b$. By the continuity of $\|T(\sigma(\cdot))\|$, there exists a sub-interval
$\mathopen{[}s_{0},s_{1}\mathclose{]}\subseteq\mathopen{[}0,1\mathclose{]}$ such that
$\|T(\sigma(\mathopen{[}s_{0},s_{1}\mathclose{]})\|=\mathopen{[}a,b\mathclose{]}$. In particular, the sub-path
$\gamma:=\sigma|_{\mathopen{[}s_{0},s_{1}\mathclose{]}}$ is such that $T(\gamma)\subseteq K_{a,b}$, $T(\gamma)\cap K_{a}\neq\emptyset$ and $T(\gamma)\cap
K_{b}\neq\emptyset$. We deduce the validity of $(ii)$, hence the claim.

\smallskip

\noindent
\textit{Compressive form}.
Fix $\varepsilon>0$ such that $a-\varepsilon>0$. Define $a'=a-\varepsilon$ and $b'=b+\varepsilon$. Set $M'=K_{a',b'}$ and $\phi:=r\circ T|_{M}$, where $r\colon K\to K$ is defined
as follows
\begin{equation*}
r(z) = \begin{cases}
\, r_{a'}(z), & \text{if } z\in K_{0,a'}; \\ 
\, z,         & \text{if } z\in K_{a',b'}; \\
\, r_{b'}(z), & \text{if } z\in K\setminus K_{0,b'}; 
\end{cases}
\end{equation*}
and $r_{a'},r_{b'}$ are the retractions defined in Example \ref{exampleK}.
Using Corollary~\ref{ThCompressiveHomeomorphism} we obtain the thesis. 
\end{proof}

{} From what has been discussed in Example~\ref{exampleK}, one can easily generalize Theorem~\ref{Krasnosel'skii normed fixed point theorem} by the following result.
For other versions of Krasnosel'ski\u{\i} fixed point theorem, we refer to \cite{GuLa-1988,LeWi-1979,Sc-1976}.

\begin{theorem}\label{Krasnosel'skii l fixed point theorem}
Let $(X,\|\cdot\|)$ be a normed linear space and $K$ a cone in $X$. Let $l\colon X\to\mathbb{R}$ be a continuous functional, 
strictly positive and positively homogeneous. 
Fix $a,b\in\mathbb{R}^{+}$ with $a<b$. Define $K_{a,b}^{l}:=\{x\in K\colon a\leq l(x)\leq b\}$ and $K_{d}^{l}:=\{x\in K\colon l(x)=d\}$, for all $d\in\mathbb{R}^{+}$.

Consider a compact map $T\colon K\to K$.
\begin{itemize}
	\item (\textit{Expansive form}) $T$ has at least a fixed point in $K_{a,b}^{l}$ if
					\begin{itemize}
						\item [$(E_{a}^{l})$] $l(T(x))\leq a$, $\forall \, x\in K_{a}^{l}$;
						\item [$(E_{b}^{l})$] $l(T(x))\geq b$, $\forall \, x\in K_{b}^{l}$.
					\end{itemize}
	\item (\textit{Compressive form}) $T$ has at least a fixed point in $K_{a,b}^{l}$ if
					\begin{itemize}
						\item [$(C_{a}^{l})$] $l(T(x))\geq a$, $\forall \, x\in K_{a}^{l}$;
						\item [$(C_{b}^{l})$] $l(T(x))\leq b$, $\forall \, x\in K_{b}^{l}$.
					\end{itemize}
\end{itemize}
\end{theorem}

\begin{proof}
The proof is the same of Theorem~\ref{Krasnosel'skii normed fixed point theorem}, replacing $\|\cdot\|$ with $l(\cdot)$.
\end{proof}

\begin{remark}
We note that usually Krasnosel'ski\u{\i} fixed point theorems are set on Banach spaces, while our versions are also valid on normed spaces.
Furthermore, it is typically required that the neighbourhoods of the origin are bounded and that the operator is completely continuous (which is
equivalent to the compactness, if the neighbourhoods are bounded). 
We remark that hypotheses of Theorem~\ref{Krasnosel'skii normed fixed point theorem} and Theorem~\ref{Krasnosel'skii l fixed point theorem}
do not require the boundedness of the set $\{x\in K\colon l(x)\leq b\}$, but we demand the compactness of the operator $T$.
\end{remark}

We end this section by considering other forms for the domain of the operator $T$. We start with 
the following theorem, still presented in \cite{Kw-2008fpta,Kw-2008} and originally set in a Banach space. See also \cite[Corollary~6.3, p.~452]{GrDu-2003}.

\begin{theorem}
\label{Kwong L fixed point theorem}
Let $(X,\|\cdot\|)$ be a normed linear space and $(B[x_{i},r_{i}])_{i}$ a collection of $2n$ closed balls pairwise disjoint contained in $B[0,R]$.
Define $L:=B[0,R]\setminus\bigcup_{i=1}^{2n}B(x_{i},r_{i})$ and consider a compact map $T\colon L\to B[0,R]$ such that $T(\partial B(x_{i},r_{i}))\subseteq
B[x_{i},r_{i}]$ for all $i\in\{1,\ldots,n\}$. Then $T$ has at least a fixed point in $L$.
\end{theorem}

Using the fixed point index, the proof is quite straightforward: assuming that there are no fixed points on $\partial L$, it is proved that
$i_{X}(T,\text{\rm int}(L))=1-2n$, using the additivity property of the fixed point index.

{} From the computation of $i_{X}(T,\text{\rm int}(L))=1-2n$, we immediately see that the assumptions of the theorem are very restrictive; in fact, to
have a fixed point in $L$, it is sufficient to assume that the number of holes is different from 1. 
Moreover, if we suppose that this number is equal to $1$, the constant map $T\equiv x_{1}$ defines a possible counterexample: clearly $T$ satisfies every
hypothesis, but it has no fixed points in $L$.

The following result shows a possible solution for the case of a ball with a single hole in an infinite dimensional normed space. 
Assuming in addition to the hypotheses of Theorem~\ref{Kwong L fixed point theorem} that $\partial B(0,R)$ is invariant (i.e.~$T(\partial B(0,R))\subseteq \partial B(0,R)$),
the assertion with a single hole is valid, as shown in the expansive form of this theorem.

\begin{theorem}
\label{annulus fixed point theorem}
Let $(X,\|\cdot\|)$ be an infinite dimensional normed linear space. Fix $r_{1},r_{2}\in\mathbb{R}$ such that $0<r_{1}<r_{2}$ and define the annulus
\begin{equation*}
 A:=\{x\in X\colon r_{1}\leq\|x\|\leq r_{2}\}
\end{equation*}
and $A_{d}:=\partial B(0,d)=\{x\in X\colon \|x\|=d\}$, for all $d\in\mathbb{R}^{+}$.
Consider a compact map $T\colon A\to X$.
\begin{itemize}
	\item (\textit{Expansive form}) $T$ has at least a fixed point in $A$ if
					\begin{itemize}
						\item [$(E_{a})$] $\|T(x)\|\leq\|x\|$, $\forall \, x\in A_{r_{1}}$;
						\item [$(E_{b})$] $\|T(x)\|\geq\|x\|$, $\forall \, x\in A_{r_{2}}$.
					\end{itemize}
	\item (\textit{Compressive form}) $T$ has at least a fixed point in $A$ if
					\begin{itemize}
						\item [$(C_{a})$] $\|T(x)\|\geq\|x\|$, $\forall \, x\in A_{r_{1}}$;
						\item [$(C_{b})$] $\|T(x)\|\leq\|x\|$, $\forall \, x\in A_{r_{2}}$.
					\end{itemize}
\end{itemize}
\end{theorem}

\begin{proof}
The proof is the same of Theorem~\ref{Krasnosel'skii normed fixed point theorem}, observing that $A_{1}$ is an AR in an infinite dimensional normed linear space 
(cf.~\cite[Theorem~6.2]{Du-1951} or \cite[Theorem~7.7, p.~164--165]{GrDu-2003}).
\end{proof}

The statement is false in finite dimension. For example, it is sufficient to consider a nontrivial rotation $\vartheta$ of centre $0$ in $\mathbb{R}^{2}$. 
Clearly $\vartheta$ is continuous (hence compact) and it has no fixed points.

\begin{remark}
We underline that Theorem~\ref{Kwong L fixed point theorem} remains valid 
if we consider sets homeomorphic to the ball at the place of $B(0,R)$ and $B(x_{i},r_{i})$.
\end{remark}

\begin{remark}
Concerning Theorem~\ref{Kwong L fixed point theorem} set in infinite dimensional spaces, if we replace the assumption $T(\partial B(x_{i},r_{i}))\subseteq
B[x_{i},r_{i}]$, for all $i\in\{1,\ldots,n\}$, with $T(\partial B(x_{i},r_{i}))\subseteq L$, for all $i\in\{1,\ldots,n\}$, we obtain the same thesis by applying
Generalized Schauder Theorem (notice that $L$ is an AR, see~\cite[Theorem~6.2]{Du-1951} or \cite[Theorem~7.7, p.~164--165]{GrDu-2003}).
\end{remark}

\section{Appendix}

In this final section we present the fixed point index defined on sets contained in ANRs. In particular, 
we list the axioms and the main properties which are relevant for this paper.
For more details and proofs, we refer to \cite{Am-1976,Gr-1972,GrDu-2003,Ma-1999,Nu-1985,Nu-1993}
and the references therein. 

Let $X$ be an ANR and $U\subseteq X$ an open subset. Consider a continuous map $f\colon U\to X$
such that $\text{Fix}(f):=\{x\in U\colon f(x)=x\}$ is a compact set (possibly empty) 
and such that there exists an open neighbourhood $V$ of $\text{Fix}(f)$ with $\overline{V}\subseteq U$ such that $f|_{\overline{V}}$ is compact.
If all the previous assumptions are satisfied, the triplet $(X,U,f)$ is said to be \textit{admissible}.

To an admissible triplet $(X,U,f)$ we associate an integer
\begin{equation*}
i_{X}(f,U),
\end{equation*}
called the \textit{fixed point index of $f$ on $U$ relatively to $X$}, satisfying the following properties.

\begin{enumerate}
								
	\item [$(a)$] \textit{Additivity}. If $U_{1},U_{2}\subseteq U$ are open and disjoint subsets and $\text{Fix}(f)\subseteq U_{1}\cup U_{2}$, then
							\begin{equation*}
							i_{X}(f,U)=i_{X}(f,U_{1})+i_{X}(f,U_{2}).
							\end{equation*}

	\item [$(b)$] \textit{Homotopy invariance}. Let $h\colon\mathopen{[}0,1\mathclose{]}\times U \to X$, $h_{\lambda}(x):=h(\lambda,x)$,
							be a continuous homotopy such that	
							$\Sigma:=\bigcup_{\lambda\in\mathopen{[}0,1\mathclose{]}}\{x\in U\colon x-h_{\lambda}(x)=0\}$ 
							is a compact set and there exists an open neighbourhood $V$ of $\Sigma$ such that $\overline{V}\subseteq U$ and
							$h|_{\mathopen{[}0,1\mathclose{]}\times\overline{V}}$ is a compact map. 
							Then $i_{X}(h_{\lambda},U)$ is constant with respect to $\lambda\in\mathopen{[}0,1\mathclose{]}$.

	\item [$(c)$] \textit{Weak normalization}. Let $f(x)=p$, $\forall \, x\in U$, then
							\begin{equation*}
							i_{X}(f,U):=\begin{cases}
							\, 1, & \text{if } p\in U;\\
							\, 0, & \text{if } p\notin U.
							\end{cases}
							\end{equation*}

	\item [$(d)$] \textit{Strong normalization}. If $U=X$ and the map $f$ is compact, then the Lefschetz number of $f$ is defined and
							\begin{equation*}
							i_{X}(f,U)=\Lambda(f).
							\end{equation*}

	\item [$(e)$] \textit{Fixed point property}. If $i_{X}(f,U)\neq0$, then $\text{Fix}(f)\neq\emptyset$, i.e.~$f$ has a fixed point.

	\item [$(f)$] \textit{Excision}. Let $U_{0}$ be an open subset of $U$ such that $\text{Fix}(f)\subseteq U_{0}$. Then
							\begin{equation*}
							i_{X}(f,U)=i_{X}(f,U_{0}).
							\end{equation*}

	\item [$(g)$] \textit{Multiplicativity}. If the triplets $(X_{1},U_{1},f_{1}),(X_{2},U_{2},f_{2})$ are admissible, then
							\begin{equation*}
							i_{X_{1}\times X_{2}}(f_{1}\times f_{2},U_{1}\times U_{2})=i_{X_{1}}(f_{1},U_{1})\cdot i_{X_{2}}(f_{2},U_{2}),
							\end{equation*}
							where $f_{1}\times f_{2}\colon U_{1}\times U_{2}\to X_{1}\times X_{2}$, $(f_{1}\times f_{2})(u,v):=(f_{1}(u),f_{2}(v))$.

	\item [$(h)$] \textit{Commutativity}. Let $U_{1},U_{2}$ be open subsets of the ANRs $X_{1},X_{2}$ respectively.
							Suppose that $f_{1}\colon U_{1}\to X_{2}$, $f_{2}\colon U_{2}\to X_{1}$ are continuous maps and $f_{1}$ is compact 
							in a neighbourhood of
							$\{x\in U_{1}\colon f_{2}f_{1}(x)=x\}$ (or $f_{2}$ is compact in a neighbourhood of
							$\{x\in U_{2}\colon f_{1}f_{2}(x)=x\}$), where $f_{2}f_{1}\colon f_{1}^{-1}(U_{2})\to X_{1}$ and 
							$f_{1}f_{2}\colon f_{2}^{-1}(U_{1})\to X_{2}$.
							
							If the triplets $(f_{1}^{-1}(U_{2}),U_{1},f_{2}f_{1})$, $(f_{2}^{-1}(U_{1}),U_{2},f_{1}f_{2})$ are admissible,
							then
							\begin{equation*}
							i_{X_{1}}(f_{2}f_{1},f_{1}^{-1}(U_{2}))=i_{X_{2}}(f_{1}f_{2},f_{2}^{-1}(U_{1})).
							\end{equation*}

	\item [$(i)$] \textit{Contraction}. Let $Y\subseteq X$ be an ANR such that the inclusion $j\colon Y\to X$ is continuous
							and $f(U)\subseteq Y$. Then	
							\begin{equation*}
							i_{X}(f,U)=i_{Y}(f,U\cap Y).
							\end{equation*}

	\item [$(j)$] \textit{Localization}. Let $U_{1},U_{2}\subseteq U$ be open and disjoint subsets and $\text{Fix}(f)\subseteq U_{1}\cup U_{2}$. 
							If $i_{X}(f,U)\neq 0$ and $i_{X}(f,U_{1})=0$, then $\text{Fix}(f|_{U_{2}})\neq\emptyset$.

	\item [$(k)$] \textit{Multiplicity}. Let $U_{1},U_{2}\subseteq U$ be open and disjoint subsets and $\text{Fix}(f)\subseteq U_{1}\cup U_{2}$. 
							If $i_{X}(f,U)= 0$ and $i_{X}(f,U_{1})\neq 0$, 
							then $\text{Fix}(f|_{U_{1}})\neq\emptyset$ and $\text{Fix}(f|_{U_{2}})\neq\emptyset$.
	
\end{enumerate}

\begin{remark}
It is obvious that the axioms listed above are not independent. For example, properties $(j)$ and $(k)$ are direct consequence of axioms $(a)$ and $(e)$, or further $(h)$ implies $(i)$.
\end{remark}

\begin{remark}
For the application shown in this paper, it is sufficient to give the axioms of the fixed point index in a less abstract framework. If we just assume that
$U$ is an open subset of $X$ and $f\colon\overline{U}\to X$ is a compact map which has no fixed points on $\partial U$, then the triplet $(X,U,f)$ is
admissible.

In that framework the homotopy property can be written in this easier way. Let $X$ be an ANR and $U\subseteq X$ an open subset. 
Let $h\colon\mathopen{[}a,b\mathclose{]}\times\overline{U}\to X$ be a compact map such that $\text{Fix}(h_{\lambda})\cap\partial U =\emptyset$, 
for all $\lambda\in\mathopen{[}a,b\mathclose{]}$. Then $i_{X}(h_{\lambda},U)$ is constant with respect to $\lambda\in\mathopen{[}a,b\mathclose{]}$.
The remaining axioms remain almost unchanged.
\end{remark}

We conclude the appendix with a theorem which is crucial in the proof of the main result. 
It is stated in a less general framework than the one in which we have defined the fixed point index.

\begin{theorem}[Leray-Schauder continuation principle]\label{LScont_prin}
Let $X$ be an ANR, $U$ an open subset of $\mathopen{[}a,b\mathclose{]}\times X$, and let $\psi\colon \overline{U}\to X$ be a compact map such that
\begin{itemize}
 \item [$(i)$] the sets $\partial U$ and $\Sigma:=\{(t,x)\in \overline{U} \colon \psi(t,x)=x\}$ are disjoint;
 \item [$(ii)$] $i_{X}(\psi(a,\cdot), U_{a})\neq 0$, where $U_{a}=\{x\in X\colon (a,x)\in U\}$.
\end{itemize}
Then there exists a continuum $S\subseteq \Sigma$ joining the sets
\begin{equation*}
A=(X\times \{a\})\cap \Sigma  \quad \text{and} \quad B=(X\times \{b\})\cap \Sigma.
\end{equation*}
\end{theorem}

\section*{Acknowledgment}

This work benefited from long enlightening discussions, helpful suggestions and encouragement of Fabio Zanolin. 
This research was supported by \textit{SISSA - International School for Advanced Studies} and \textit{Univer\-si\-t\`{a} degli Stu\-di di Udine}.

\bibliographystyle{elsart-num-sort}
\bibliography{Feltrin_biblio}

\begin{thebibliography}{10}
\expandafter\ifx\csname url\endcsname\relax
  \def\url#1{\texttt{#1}}\fi
\expandafter\ifx\csname urlprefix\endcsname\relax\def\urlprefix{URL }\fi

\bibitem{Am-1976}
H.~Amann, Fixed point equations and nonlinear eigenvalue problems in ordered
  {B}anach spaces, SIAM Rev. 18 (1976) 620--709.

\bibitem{AvAnHe-2011}
R.~Avery, D.~Anderson, J.~Henderson, Some fixed point theorems of
  {L}eggett-{W}illiams type, Rocky Mountain J. Math. 41 (2011) 371--386.

\bibitem{ChTo-2007}
J.~Chu, P.~J. Torres, Applications of {S}chauder's fixed point theorem to
  singular differential equations, Bull. Lond. Math. Soc. 39 (2007) 653--660.

\bibitem{Du-1951}
J.~Dugundji, An extension of {T}ietze's theorem, Pacific J. Math. 1 (1951)
  353--367.

\bibitem{ErHuWa-1994}
L.~H. Erbe, S.~C. Hu, H.~Wang, Multiple positive solutions of some boundary
  value problems, J. Math. Anal. Appl. 184 (1994) 640--648.

\bibitem{ErWa-1994}
L.~H. Erbe, H.~Wang, On the existence of positive solutions of ordinary
  differential equations, Proc. Amer. Math. Soc. 120 (1994) 743--748.

\bibitem{FoGi-2015pp}
A.~Fonda, P.~Gidoni, Generalizing the {P}oincar\'{e}-{M}iranda theorem: the
  avoiding cones condition, preprint (2015).

\bibitem{FrTo-2008}
D.~Franco, P.~J. Torres, Periodic solutions of singular systems without the
  strong force condition, Proc. Amer. Math. Soc. 136 (2008) 1229--1236.

\bibitem{GrKoWa-2008}
J.~R. Graef, L.~Kong, H.~Wang, Existence, multiplicity, and dependence on a
  parameter for a periodic boundary value problem, J. Differential Equations
  245 (2008) 1185--1197.

\bibitem{Gr-1972}
A.~Granas, The {L}eray-{S}chauder index and the fixed point theory for
  arbitrary {ANR}s, Bull. Soc. Math. France 100 (1972) 209--228.

\bibitem{GrDu-2003}
A.~Granas, J.~Dugundji, Fixed point theory, Springer Monographs in Mathematics,
  Springer-Verlag, New York, 2003.

\bibitem{GuLa-1988}
D.~J. Guo, V.~Lakshmikantham, Nonlinear problems in abstract cones, Academic
  Press, Inc., Boston, MA, 1988.

\bibitem{Hu-1965}
S.-t. Hu, Theory of retracts, Wayne State University Press, Detroit, 1965.

\bibitem{Ka-2000}
J.~Kampen, On fixed points of maps and iterated maps and applications,
  Nonlinear Anal. 42 (2000) 509--532.

\bibitem{Kr-1960}
M.~A. Krasnosel'ski{\u\i}, Fixed points of cone-compressing or cone-extending
  operators, Soviet Math. Dokl. 1 (1960) 1285--1288.

\bibitem{Kr-1964}
M.~A. Krasnosel'ski{\u\i}, Positive solutions of operator equations, P.
  Noordhoff Ltd. Groningen, 1964.

\bibitem{Kw-2008fpta}
M.~K. Kwong, On {K}rasnoselskii's cone fixed point theorem, Fixed Point Theory
  Appl. (2008) Art. ID 164537, 18 pp.

\bibitem{Kw-2008}
M.~K. Kwong, The topological nature of {K}rasnoselskii's cone fixed point
  theorem, Nonlinear Anal. 69 (2008) 891--897.

\bibitem{LeWi-1979}
R.~W. Leggett, L.~R. Williams, Multiple positive fixed points of nonlinear
  operators on ordered {B}anach spaces, Indiana Univ. Math. J. 28 (1979)
  673--688.

\bibitem{LeWi-1980}
R.~W. Leggett, L.~R. Williams, A fixed point theorem with application to an
  infectious disease model, J. Math. Anal. Appl. 76 (1980) 91--97.

\bibitem{Mar-2006}
M.~M. Marsh, Fixed point theorems for partially outward mappings, Topology
  Appl. 153 (2006) 3546--3554.

\bibitem{Ma-1999}
J.~Mawhin, Leray-{S}chauder degree: a half century of extensions and
  applications, Topol. Methods Nonlinear Anal. 14 (1999) 195--228.

\bibitem{Ma-2013ans}
J.~Mawhin, Variations on {P}oincar\'e-{M}iranda's theorem, Adv. Nonlinear Stud.
  13 (2013) 209--217.

\bibitem{Ma-2013}
J.~Mawhin, Variations on some finite-dimensional fixed-point theorems,
  Ukrainian Math. J. 65 (2013) 294--301.

\bibitem{Nu-1985}
R.~D. Nussbaum, The fixed point index and some applications, vol.~94 of
  S\'eminaire de Math\'ematiques Sup\'erieures [Seminar on Higher Mathematics],
  Presses de l'Universit\'e de Montr\'eal, Montreal, QC, 1985.

\bibitem{Nu-1993}
R.~D. Nussbaum, The fixed point index and fixed point theorems, in: Topological
  methods for ordinary differential equations ({M}ontecatini {T}erme, 1991),
  vol. 1537 of Lecture Notes in Math., Springer, Berlin, 1993, pp. 143--205.

\bibitem{PaZa-2007}
D.~Papini, F.~Zanolin, Some results on periodic points and chaotic dynamics
  arising from the study of the nonlinear {H}ill equations, Rend. Semin. Mat.
  Univ. Politec. Torino 65 (2007) 115--157.

\bibitem{PiZa-2005}
M.~Pireddu, F.~Zanolin, Fixed points for dissipative-repulsive systems and
  topological dynamics of mappings defined on {$N$}-dimensional cells, Adv.
  Nonlinear Stud. 5 (2005) 411--440.

\bibitem{PiZa-2007}
M.~Pireddu, F.~Zanolin, Cutting surfaces and applications to periodic points
  and chaotic-like dynamics, Topol. Methods Nonlinear Anal. 30 (2007) 279--319.

\bibitem{Pr-2007}
R.~Precup, A vector version of {K}rasnosel'ski\u\i 's fixed point theorem in
  cones and positive periodic solutions of nonlinear systems, J. Fixed Point
  Theory Appl. 2 (2007) 141--151.

\bibitem{Sc-1976}
K.~Schmitt, Fixed point and coincidence theorems with applications to nonlinear
  differential and integral equations, S{\'e}minaires de Math{\'e}matique
  Appliqu{\'e}e et M{\'e}canique, Rapport No. 97, Universit\'{e} catholique de
  Louvain, Vander, Louvain-la-Neuve, 1976.

\bibitem{To-2003}
P.~J. Torres, Existence of one-signed periodic solutions of some second-order
  differential equations via a {K}rasnoselskii fixed point theorem, J.
  Differential Equations 190 (2003) 643--662.

\end{thebibliography}

\bigskip
\begin{flushleft}

{\small{\it Preprint}}

{\small{\it March 2015}}

\end{flushleft}

\end{document}